\documentclass[12pt]{amsart} 
\usepackage{amsfonts,graphics,amsmath,amsthm,amsfonts,amscd, amssymb,amsmath,latexsym}
\usepackage{epsfig}
\usepackage{flafter}

\usepackage{mathtools}

\theoremstyle{plain}
\newtheorem{theorem}{Theorem}[section]

\newtheorem{proposition}[theorem]{Proposition}
\newtheorem{definition-lemma}[theorem]{Definition-Lemma}

\theoremstyle{definition}
\newtheorem{definition}[theorem]{Definition}
\newtheorem{remark}[theorem]{Remark}

\newtheorem*{ack*}{Acknowledgements}

\title{Some remarks on the volume of log varieties}
\author{Stefano Filipazzi}

\usepackage[english,italian]{babel}
\usepackage{graphicx,subfigure}
\usepackage[a4paper,top=30mm,bottom=30mm,left=30mm,right=30mm]{geometry}
\usepackage{tikz}
\usetikzlibrary{matrix,arrows,decorations.pathmorphing}

\usepackage{hyperref}

\usepackage[babel]{csquotes}
\usepackage[maxnames=99,style=alphabetic,backend=bibtex]{biblatex}
\newcommand{\R}{\mathbb{R}}			
\newcommand{\C}{\mathbb{C}}			
\newcommand{\Z}{\mathbb{Z}}			
\newcommand{\rar}{\rightarrow}		
\newcommand{\Xf}{\mathcal{X}}		
\newcommand{\Bf}{\mathcal{B}}		

\DeclareMathOperator{\Spec}{Spec}		
\DeclareMathOperator{\mult}{mult}		
\DeclareMathOperator{\Supp}{Supp}		
\DeclareMathOperator{\vol}{vol}		
\DeclareMathOperator{\coeff}{coeff}	

\def\O#1.{\mathcal {O}_{#1}}			
\def\pr #1.{\mathbb P^{#1}}				
\def\af #1.{\mathbb A^{#1}}				
\def\ses#1.#2.#3.{0\to #1\to #2\to #3 \to 0}		
\def\xrar#1.{\xrightarrow{#1}}			
\def\K#1.{K_{#1}}						
\def\bA#1.{\mathbf{A}_{#1}}				
\def\bM#1.{\mathbf{M}_{#1}}				
\def\bL#1.{\mathbf{L}_{#1}}				
\def\bB#1.{\mathbf{B}_{#1}}				
\def\bK#1.{\mathbf{K}_{#1}}				
\def\subs#1.{_{#1}}						
\def\sups#1.{^{#1}}						

\usepackage{lipsum}

\newcommand{\Addresses}{{
  \bigskip
  \footnotesize

  S.~Filipazzi, \textsc{Department of Mathematics, University of Utah,
    Salt Lake City,\\ UT 84112, USA}\par\nopagebreak
  \textit{E-mail address}: \texttt{filipazz@math.utah.edu}
  
}}

\begin{document}

\selectlanguage{english}

\begin{abstract}
In this note, using methods introduced by Hacon, McKernan and Xu \cite{HMX}, we study the accumulation points of volumes of varieties of log general type. First, we show that, if the set of boundary coefficients $\Lambda$ is DCC, closed under limits and contains 1, then also the corresponding set of volumes is DCC and closed under limits. Then, we consider the case of $\epsilon$-log canonical varieties, for $0 < \epsilon < 1$. In this situation, we prove that, if $\Lambda$ is finite, then the corresponding set of volumes is discrete.
\end{abstract}


\maketitle

\tableofcontents



\section{Introduction}

Consider a projective log canonical pair $(X,B)$ of log general type. One of the main birational invariants attached to it is the volume $\vol(X,\K X. + B)$. The study of the values that the volume can attain has deep connections with the study of the boundedness of projective pairs of log general type and their possible semi-log canonical degenerations. In this direction, the main result is the following, first proved by Alexeev in the case of surfaces \cite{Ale94}, and then established in the general case by Hacon, McKernan and Xu \cite{HMX13}.

\begin{theorem}[{\cite[Theorem 1.3.1]{HMX13}}] \label{•}
Fix a positive integer $d$ and a set $\Lambda \subset [0,1]$ which satisfies the DCC. Let $\mathfrak{D}(d,\Lambda)$ be the set of projective log canonical pairs $(X,B)$ such that the dimension of $X$ is $d$ and the coefficients of $B$ belong to $\Lambda$. Then, the set
$$
\lbrace \vol(X,\K X. + B)|(X,B) \in \mathfrak{D}(d,\Lambda) \rbrace
$$
also satisfies the DCC.
\end{theorem}

In recent years, there has been a lot of activity in trying to have a better understanding of the behavior of the volume function for some particular choices of $\Lambda$. In particular, in the case of surfaces, there has been interesting progress towards the study of the minimum and the accumulation points of the set of volumes \cite{AL16,Liu17,AL18,UY}.

The aim of this note is to provide a qualitative description of the behavior of accumulation points in arbitrary dimension. Our first result can be seen as a generalization of \cite[Corollary 1.3]{AL18}.

\begin{theorem} \label{teo accumulation}
Fix a non-negative integer $d$, and a DCC set $\Lambda \subset [0,1]$. Further assume that $1 \in \Lambda$, and that $\overline{\Lambda}= \Lambda$. Denote by $\mathfrak{D}(d,\Lambda)$ the set of $d$-dimensional projective log canonical pairs $(X,B)$ with $\coeff(B) \subset \Lambda$. Also, define
$$
\mathfrak{V}(d,\Lambda) \coloneqq \lbrace \vol(X,\K X. +B)| (X,B) \in \mathfrak{D}(d,\Lambda) \rbrace.
$$
Then, we have $\overline{\mathfrak{V}(d,\Lambda)}=\mathfrak{V}(d,\Lambda)$.
\end{theorem}

In Theorem \ref{teo accumulation}, in order to achieve closedness we have to require that the coefficients are closed under limits and that $1 \in \Lambda$. Our second result shows that these two phenomena are the only sources of accumulation points for the volume function. More precisely, we prove the following.

\begin{theorem} \label{teo discrete}
Fix a positive number $0 < \epsilon < 1$, a non-negative integer $d$, and a finite set $\Lambda \subset [0,1]$. Denote by $\mathfrak{D}(d,\Lambda,\epsilon)$ the set of $d$-dimensional projective $\epsilon$-log canonical pairs $(X, B)$ with $\coeff(B) \subset \Lambda$. Also, define
$$
\mathfrak{V}(d,\Lambda,\epsilon) \coloneqq \lbrace \vol(X,\K X. +B)| (X,B) \in \mathfrak{D}(d,\Lambda,\epsilon) \rbrace.
$$
Then, $\mathfrak{V}(d,\Lambda,\epsilon)$ is a discrete set.
\end{theorem}

The proofs of both Theorem \ref{teo accumulation} and Theorem \ref{teo discrete} rely on a careful analysis of arguments by Hacon, McKernan and Xu \cite{HMX}. In particular, one first restricts the attention to a fixed birational class and proves the desired statement in this special setup. Then, by log birational boundedness and deformation invariance of plurigenera, one can deduce the complete statement from the special case.

After the first version of this work was completed, Chen Jiang showed us a different and direct approach to Theorem \ref{teo discrete}. This strategy is discussed in Remark \ref{remark alternative}.

\begin{ack*}
The author would like to thank his advisor Christopher Hacon for bringing these questions to his attention. He would also like to thank Chen Jiang for pointing out an alternative strategy to prove Theorem \ref{teo discrete}.
Finally, he would like to express his gratitude to the anonymous referee for the careful report and the many suggestions.
The author was partially supported by NSF research grants no: DMS-1300750, DMS-1265285 and by a grant from the Simons Foundation; Award Number: 256202.
\end{ack*}

\section{Preliminaries}

In this paper, we work over the field of complex numbers $\C$. In particular, all the constructions and statements have to be understood in this setting. Now, we recall a few definitions that will be relevant in the following.

\begin{definition}
Let $X$ be a normal projective variety. A {\it boundary} $B$ is an effective $\R$-divisor with coefficients in $[0,1]$ such that $K_X+B$ is $\R$-Cartier. We denote by $\coeff(B)$ the {\it set of coefficients} used to write $B$ as a combination of prime Weil divisors. A {\it log pair} (or simply a {\it pair}) $(X,B)$ is the datum of a normal projective variety and a boundary.
\end{definition}

Now, we recall the classic measure of singularities for log pairs. Consider a log pair $(X,B)$, and let $f:Y \rightarrow X$ be a proper birational morphism from a normal variety $Y$. Choose the canonical divisor $\K Y.$ such that $f_* \K Y. = \K X.$. Then, we define a divisor $B_Y$ by the equation
$$
K_Y+B_Y \coloneqq f^*(K_X+B).
$$
Fix $0\leq \epsilon \leq 1$. We say that $(X,B)$ is {\it $\epsilon$-log canonical} (respectively {\it Kawamata log terminal}) if $\mult_P B_Y \leq 1- \epsilon$ (respectively $\mult_P B_Y < 1$) for every prime divisor $P \subset Y$ and every $f:Y \rar X$ as above. If $\epsilon=0$, we drop it from the notation, and say that $(X,B)$ is log canonical.

Now, we recall the language of b-divisors, as they come up naturally in the setup of the proofs in this note.

\begin{definition}
Let $X$ be a normal variety, and consider the set of all proper birational morphisms $\pi: X_\pi \rightarrow X$, where $X_\pi$ is normal. This is a partially ordered set, where $\pi' \geq \pi$ if $\pi'$ factors through $\pi$. We define the space of {\it Weil b-divisors} as the inverse limit
$$
\mathbf{Div}(X)\coloneqq \varprojlim_\pi \mathrm{Div}(X_\pi),
$$
where $\mathrm{Div}(X_\pi)$ denotes the space of Weil divisors on $X_\pi$. Then, we define the space of {\it $\R$-Weil b-divisors} as $\mathbf{Div}(X)_\R \coloneqq \mathbf{Div}(X)\otimes_\Z \R$.
\end{definition}

For a more detailed discussion of b-divisors in this setting, see \cite{HMX}. Here, we will just recall the constructions that will appear subsequently. Consider a log pair $(X,B)$. Fix $0 \leq \epsilon \leq 1$. Since a b-divisor $\mathbf{D}$ is determined by its corresponding traces $\mathbf{D}_Y$ on each birational model $Y \rightarrow X$, we can define b-divisors $\mathbf{M}_B^\epsilon$ and $\mathbf{L}_B$ as follows. For every birational morphism $\pi : Y \rightarrow X$, define $B_Y$ by $K_Y+B_Y : =\pi^*(K_X+B)$. Then, we declare
$$
\mathbf{M}^\epsilon_{B,Y} \coloneqq \pi_*^{-1}B+(1-\epsilon)\mathrm{Ex}(\pi), \quad \mathbf{L}_{B,Y} \coloneqq B_Y \vee 0.
$$
Here $\vee$ denotes the maximum between two divisors; similarly, we will use $\wedge$ to denote the minimum between two divisors. In case $\epsilon=0$, we drop it from the notation and write $\mathbf{M}_B \coloneqq \mathbf{M}_B^0$. The b-divisor $\mathbf{M}_B$ encodes the information of the strict transforms and the exceptional divisors on all higher models, while $\mathbf{L}_B$ records the effective part of the sub-boundaries obtained by pulling back $K_X+B$ to any higher model.

As last piece of terminology, we recall the notions of {\it boundedness} and {\it log birational boundedness}.

\begin{definition} \label{definizione}
Let $\mathfrak{D}$ be a set of log pairs. We say that $\mathfrak{D}$ is log bounded (respectively log birationally bounded) if there is a log pair $(\Xf,\Bf)$ with $\Bf$ reduced, and there is a projective morphism $\Xf \rar T$, where $T$ is of finite type, such that for every log pair $(X,B) \in \mathfrak{D}$ there is a closed point $t \in T$ and a morphism $f : \Xf_t \rar X$ inducing an isomorphism $(X,B \subs \mathrm{red}.) \cong (\Xf_t,\Bf_t)$ (respectively, a birational map such that the support of $\Bf_t$ contains the support of the strict transform of $B$ and of any $f$-exceptional divisor).
\end{definition}

\begin{remark}
Let $\mathfrak{D}$ be a log (birationally) bounded family of pairs, and let $\mathcal{X} \rightarrow T$ be a bounding family.
Up to stratifying the base, we can assume that $T$ is smooth.
\end{remark}

\section{Proof of main results}

In this section, we prove the main results of this note. First, we focus on a special case of Theorem \ref{teo accumulation}. We say that a pair $(X,B)$ with a morphism $X \rar U$ is called {\it log smooth over} $U$ if $U$ is smooth, $B$ has simple normal crossing support, and every stratum of $(X,\Supp(B))$, including $X$, is smooth over $U$. In case $U = \Spec \C$, we just say that $(X,B)$ is log smooth. The following is a slight generalization of \cite[Proposition 4.1]{HMX}, and its proof goes through almost verbatim.

\begin{proposition} \label{model where ineq descends}
Fix a positive number $v$, a non-negative integer $d$, and a DCC set $\Lambda \subset [0,1]$. Let $(Z,D)$ be a projective log smooth $d$-dimensional pair where $D$ is reduced. Then, there exists $f:Z' \rightarrow Z$, a finite sequence of blow-ups along strata of the b-divisor $\mathbf{M}_D$, such that if
\begin{itemize}
\item $(X,B)$ is a projective log smooth $d$-dimensional pair;
\item $g : X \rightarrow Z$ is a finite sequence of blow-ups along strata of $\mathbf{M}_D$;
\item $\coeff (B) \subset \Lambda$;
\item $g_*B \leq D$;
\item $\vol (X,K_X+B) \leq v$;
\end{itemize}
then $\vol(Z',K_{Z'}+\mathbf{M}_{B,Z'}) \leq v$.
\end{proposition}

\begin{proof}
We can assume that $1 \in \Lambda$. Let $\mathfrak{P}$ be the set of pairs $(X,B)$ over $(Z,D)$ that satisfy the first four of the five hypotheses of the statement. Then, define
\begin{equation*}
\mathfrak{W}  \coloneqq  \lbrace \vol(X,K_X+B)|(X,B)\in \mathfrak{P} \rbrace.
\end{equation*}
By \cite[Theorem 3.0.1]{HMX}, $\mathfrak{W}$ satisfies the DCC. Therefore, there is a constant $\delta > 0$ such that, if $\vol(X,K_X+B) \leq v + \delta$, then $\vol(X,K_X+B) \leq v$. Also, by \cite[Theorem 3.2.1]{HMX}, there exists an integer $r$ such that, if $(X,B) \in \mathfrak{P}$ and $K_X+B$ is big, then $K_X+\frac{r-1}{r}B$ is big as well. Now, fix $\sigma > 0$ such that $(1-\sigma)^d> \frac{v}{v+\delta}$, and define $a  \coloneqq  1-\frac{\sigma}{r}$.

Then, we have the following chain of inequalities
\begin{equation} \label{equation with fractions}
\begin{split}
\vol(X,K_X+aB) &\geq \vol (X,(1-\sigma)(K_X+B))\\
&=(1-\sigma)^d \vol(X,K_X+B)\\
&> \frac{v}{v+\delta}\vol(X,K_X+B),
\end{split}
\end{equation}
where the first inequality comes from the identity
\begin{equation*}
K_X+aB= (1-\sigma)(K_X+B)+\sigma \left( K_X+\frac{r-1}{r}B  \right) 
\end{equation*}
and $K_X+\frac{r-1}{r}B$ being big.

Now, since $(Z,aD)$ is Kawamata log terminal, we can obtain a terminalization $f:Z'\rightarrow Z$ by blowing up strata of $\mathbf{M}_D$. We can write
\begin{equation*}
K_{Z'}+\Psi = f^*(K_Z+aD)+E,
\end{equation*}
where $\Psi$ and $E$ are effective, $\Psi \wedge E =0$, and $(Z',\Psi)$ is terminal.

Let $\mathfrak{F}$ denote the set of pairs $(X,B)$ satisfying all the assumptions in the statement, and such that the rational map $\phi: X \dashrightarrow Z'$ is a morphism. Fix $(X,B) \in \mathfrak{F}$, and define $B_{Z'} \coloneqq \phi_*B$. Then, by construction, we have $f_*(aB_{Z'})\leq aD$. Therefore, if we write
\begin{equation*}
K_{Z'}+\Phi= f^*(K_Z +f_* (aB_{Z'}))+F,
\end{equation*}
where $\Phi$ and $F$ are effective with $\Phi \wedge F =0$, then $(Z',\Phi)$ is terminal. Hence, it follows that
\begin{equation} \label{equationTerminalization1}
\begin{split}
\vol(Z',K_{Z'}+aB_{Z'}) &= \vol(Z',K_{Z'}+aB_{Z'}\wedge \Phi)\\
&= \vol(X,K_{X}+\phi_*^{-1}(aB_{Z'}\wedge \Phi) )\\
& \leq \vol(X,K_{X}+B),
\end{split}
\end{equation}
where the first equality follows from part (3) of \cite[Lemma 1.5.1]{HMX}, the second one from part (2) of \cite[Lemma 1.5.1]{HMX} and $(Z',aB_{Z'}\wedge \Phi)$ being terminal, and the last inequality from $\phi_*^{-1}(aB_{Z'}\wedge \Phi) \leq B$.

Then, we get the following chain of inequalities
\begin{equation*} \label{equationTerminalization2}
\vol(Z',K_{Z'}+B_{Z'}) \leq \frac{v + \delta}{v} \vol(Z',K_{Z'}+aB_{Z'}) \leq v+\delta,
\end{equation*}
where the first one follows from inequality \eqref{equation with fractions} and the second one from inequality \eqref{equationTerminalization1}. Therefore, by definition of $\delta$, we have $\vol(Z',K_{Z'}+B_{Z'}) \leq v$.

To conclude the proof, it is enough to notice that, if $(X,B)$ satisfies the assumptions in the statement, then after blowing up along finitely many strata of $\mathbf{M}_D$ and replacing $B$ by its strict transform plus the exceptional divisors, we may assume that $(X,B) \in \mathfrak{F}$.
\end{proof}

Now, we are ready to prove Theorem \ref{teo accumulation}.

\begin{proof}[Proof of Theorem \ref{teo accumulation}]
By \cite[Theorem 3.0.1]{HMX}, we know that $\mathfrak{V}(d,\Lambda)$ is a DCC set. Therefore, to check that $\overline{\mathfrak{V}(d,\Lambda)}=\mathfrak{V}(d,\Lambda)$, we are left with considering sequences of volumes that are strictly increasing and bounded from above.

Thus, let $\lbrace v_i \rbrace \subs i \geq 1. \subset \mathfrak{V}(d,\Lambda)$ be a strictly increasing sequence with limit $v < + \infty$. Our goal is to show that $v \in \mathfrak{V}(d,\Lambda)$. Also, let $\lbrace (X_i,B_i) \rbrace \subs i \geq 1. \subset \mathfrak{D}(d,\Lambda)$ be a sequence of pairs such that $\vol(X_i,\K X_i. + B_i)=v_i$.

By \cite[Theorem 3.0.1]{HMX} and \cite[Proposition 1.10.4]{HMX}, the pairs $(X_i,B_i)$ are log birationally bounded.
Up to passing to a subsequence, we may assume that there exists a bounding family $(\Xf,\Bf) \rar T$ that is log smooth over $T$ with $T$ irreducible. We will write $(\Xf_{i},\Bf_{i})$ for the pair corresponding to $(X_i,\Supp(B_i))$ in the sense of Definition \ref{definizione}.

Now, by \cite[Proposition 4.1]{HMX}, for each $i$ there exists a sequence of blow-ups along strata of $\mathbf{M}_{\mathcal{B}_i}$, denoted by $\mathcal{W}_i\sups (i). \rightarrow \mathcal{X}_i$, satisfying the following property: Given any higher model $(\tilde{\mathcal{X}}_i,\tilde{B}) \xrightarrow{\alpha} \mathcal{X}_i$ obtained by blow-ups along strata of $\mathbf{M}_{\mathcal{B}_i}$ and such that $\alpha_ * \tilde{B} \leq \mathcal{B}_i$, $\tilde{B} \in \Lambda$ and $\vol(\tilde{\mathcal{X}}_i,K_{\tilde{\mathcal{X}}_i}+\tilde{B}) = v_i$, then we have the equality $\vol(\mathcal{W}_i \sups (i).,K_{\mathcal{W}_i\sups (i).}+\mathbf{M}_{\tilde{B},\mathcal{W}_i\sups (i).}) = v_i$. 

Since $(\mathcal{X},\mathcal{B})$ is log smooth over $T$, we can assume that $\mathcal{W}_i \sups (i).$ appears as fiber of a sequence of blow-ups $p: \mathcal{W}\sups (i).\rightarrow \mathcal{X}$ along strata of $\mathbf{M}_{\mathcal{B}}$. Then, let $\Psi \sups (i).$ be the unique divisor supported on $\mathbf{M}_{\mathcal{B},\mathcal{W} \sups (i).}$ such that $\Psi \sups (i)._i = \bM B_i, \mathcal{W}_i^{(i)}.$. Then, we have
$$
v_i=\vol(\mathcal{W}\sups(i)._i,\K \mathcal{W}^{(i)}_i. + \Psi \sups (i)._i) = \vol(\mathcal{W}\sups(i)._1,\K \mathcal{W}^{(i)}_1. + \Psi \sups (i)._1).
$$
Here the second equality holds by \cite[Theorem 1.9.2]{HMX}, while the first one follows from \cite[proof of Corollary 2.1.3]{HMX}.

Thus, up to replacing $(X_i,B_i)$ with $(\mathcal{W}\sups(i)._1,\Psi \sups (i)._1)$, we may assume that we are in the situation of Proposition \ref{model where ineq descends}. In particular, $(\Xf_1,\Bf_1)$ plays the role of the model $(Z,D)$ appearing in the statement. By abusing notation, we will denote $(\Xf_1,\Bf_1)$ by $(Z,D)$.

Now, let $Z' \rightarrow Z$ be the sequence of blow-ups along strata of $\mathbf{M}_{D}$ constructed in Proposition \ref{model where ineq descends} relative to the volume $v$. Then, by Proposition \ref{model where ineq descends} and part (1) of \cite[Lemma 1.5.1]{HMX}, we have
$$
v_i \leq \vol(Z',\K Z'. + \bM B_i,Z'.) \leq v
$$
for all $i \geq 1$.

As the coefficients of $\bM B_i,Z'.$ are in the DCC set $\Lambda$, up to passing to a subsequence, we may assume $\bM B_i,Z'. \leq \bM B_{i+1},Z'.$ for all $i \geq 1$. Thus, we have a well defined limit divisor $B \subs \infty. \coloneqq \lim \subs i \to + \infty. \bM B_i,Z'.$. Since $\Lambda = \overline{\Lambda}$, we have $\coeff(B \subs \infty.) \subset \Lambda$.

Now, as $B \subs \infty. \geq \bM B_i,Z'.$, we have $v_i \leq \vol(Z',\K Z'. + B \subs \infty.)$ for all $i \geq 1$. Also, by continuity of the volume function \cite[Theorem 2.2.44]{LAZ1}, the inequality $\vol(Z',\K Z'. + B \subs \infty.) \leq v$ holds. Then, since $\lim \subs i \to + \infty. v_i = v$, we conclude $\vol(Z',\K Z'. + B \subs \infty.)=v$. Thus, the claim follows.
\end{proof}

Now, we change the focus and move towards the proof of Theorem \ref{teo discrete}. The strategy of the proof is similar to the one of Theorem \ref{teo accumulation}. Thus, we will start by proving a special version of the statement.

\begin{proposition} \label{finite over fixed model}
Fix a non-negative number $0 \leq \epsilon < 1$, a non-negative integer $d$, and a finite set $\Lambda \subset [0,1]$. Let $(Z,D)$ be a projective log smooth $d$-dimensional pair where $D$ is reduced. Consider the set of pairs $\mathfrak{Z}$ satisfying the following conditions:
\begin{itemize}
\item $(X,B)$ is a projective log smooth $d$-dimensional pair;
\item $f : X \rightarrow Z$ is a finite sequence of blow-ups along strata of $\mathbf{M}_D$;
\item $\coeff (B) \subset \Lambda$;
\item $f_*B \leq D$;
\item $(X,B)$ is $\epsilon$-log canonical;
\end{itemize}
then $\mathfrak{V}\coloneqq \lbrace \vol(X,K_{X}+B)|(X,B) \in \mathfrak{Z}\rbrace$ is finite.
\end{proposition}

\begin{proof}
We may assume that $1-\epsilon \in \Lambda$. Notice that no component of $B$ can have coefficient strictly greater than $1-\epsilon$. Fix $(X,B) \in \mathfrak{Z}$, and let $\hat{Z} \rar Z$ be a birational model obtained by blowing up strata of $\bM D.$. Then, by blowing up strata of $\bM D.$, we may find a higher model $\phi:\hat{X} \rar X$ that maps to $\hat{Z}$. Then, by part (3) of \cite[Lemma 1.5.1]{HMX}, we have
$$
\vol(X,\K X. + B) = \vol (\hat X,\K \hat X. + \bM B, \hat X.^\epsilon ).
$$
Thus, in order to show that $\mathfrak{V}$ is finite, we are free to restrict our attention to the subset of pairs $(X,B) \in \mathfrak{Z}$ that admit a morphism to a given higher model of $Z$. 

Let $\pi: Z' \rar Z$ be a terminalization of $(Z,(1-\epsilon)D)$ obtained by blowing up strata of $\bM D.$. Let $(X,B) \in \mathfrak{Z}$ be such that $f : X \rar Z$ factors through $g: X \rar Z'$. By part (3) of \cite[Lemma 1.5.1]{HMX}, we have
\begin{equation} \label{equation some coeff are 0}
\vol(X,\K X. + B) = \vol (X,\K X. + B \wedge \bL f_*B,X.).
\end{equation}
Since $0 \leq f_*B \leq (1-\epsilon)D$, by equation \eqref{equation some coeff are 0} we may assume that the coefficients in $B$ of the $g$-exceptional divisors are 0.

Thus, by what observed so far, we can obtain any volume in $\mathfrak{V}$ with a pair $(X,B) \in \mathfrak{Z}$ admitting a morphism $g : X \rar Z'$, and such that the coefficients in $B$ of the $g$-exceptional divisors are 0. Now, let $\nu : Z'' \rar Z'$ be a terminalization of $(Z',(1-\epsilon)\bM D,Z'.)$ obtained by blowing up strata of $\bM D.$. Furthermore, we may assume that $X$ admits a morphism $h : X \rar Z''$.

As the $g$-exceptional divisors in $B$ have 0 as coefficient, we have that $(Z'',h_*B)$ is terminal. Therefore, we have
$$
\K X. + B = h^*(\K Z''. + h_* B) + E,
$$
where $E$ is effective and $h$-exceptional. By part (2) of \cite[Lemma 1.5.1]{HMX}, we have
$$
\vol(X,\K X. + B) = \vol (Z'',\K Z''. + h_*B).
$$
Therefore, all the volumes in $\mathfrak{V}$ are computed by pairs supported by a fixed log-smooth pair, i.e. $(Z'',\bM D,Z''.)$, and having coefficients in the finite set $\Lambda$. As there are just finitely many combinations for the coefficients, we conclude that $\mathfrak{V}$ is finite.
\end{proof}

Now, we are ready to prove Theorem \ref{teo discrete}.

\begin{proof}[Proof of Theorem \ref{teo discrete}]
We may assume $1-\epsilon \in \Lambda$. By \cite[Theorem 3.0.1]{HMX}, we know that $\mathfrak{V}(d,\Lambda,\epsilon)$ is a DCC set. Therefore, it is enough to show that $\mathfrak{V}(d,\Lambda,\epsilon)$ has no accumulation points from below. Thus, assume by contradiction that there is a strictly increasing sequence $\lbrace v_i \rbrace \subs i \geq 1. \subset \mathfrak{V}(d,\Lambda,\epsilon)$ with limit $v < + \infty$. Also, let $\lbrace (X_i,B_i) \rbrace \subs i \geq 1. \subset \mathfrak{D}(d,\Lambda,\epsilon)$ be a sequence of pairs such that $\vol(X_i,\K X_i. + B_i)=v_i$.

By \cite[Theorem 3.0.1]{HMX} and \cite[Proposition 1.10.4]{HMX}, the pairs $(X_i,B_i)$ are log birationally bounded. Up to passing to a subsequence, we may assume that there exists a bounding family $(\Xf,\Bf) \rar T$ that is log smooth over $T$ with $T$ irreducible. We will write $(\Xf_{i},\Bf_{i})$ for the pair corresponding to $(X_i,\Supp(B_i))$ in the sense of Definition \ref{definizione}.

Now, by the proof of Proposition \ref{finite over fixed model}, for each $i$ there exists a sequence of blow-ups along strata of $\mathbf{M}_{\mathcal{B}_i}$, denoted by $\mathcal{W}_i\sups (i). \rightarrow \mathcal{X}_i$, satisfying the following property: Given any higher model $(\tilde{\mathcal{X}}_i,\tilde{B}) \xrightarrow{\alpha} \mathcal{X}_i$ obtained by blow-ups along strata of $\mathbf{M}_{\mathcal{B}_i}$ and such that $\alpha_ * \tilde{B} \leq \mathcal{B}_i$, with $\coeff(\tilde{B}) \subset \Lambda$, then we have
$$
\vol(\mathcal{W}_i \sups (i).,K_{\mathcal{W}_i\sups (i).}+\mathbf{M}^\epsilon_{\tilde{B},\mathcal{W}_i\sups (i).}) = \vol(\tilde{\mathcal{X}}_i,K_{\tilde{\mathcal{X}}_i}+\tilde{B}).
$$ 

Since $(\mathcal{X},\mathcal{B})$ is log smooth over $T$, we can assume that $\mathcal{W}_i \sups (i).$ appears as fiber of a sequence of blow-ups $p: \mathcal{W}\sups (i).\rightarrow \mathcal{X}$ along strata of $\mathbf{M}_{\mathcal{B}}$. Then, let $\Psi \sups (i).$ be the unique divisor supported on $\mathbf{M}_{\mathcal{B},\mathcal{W} \sups (i).}$ such that $\Psi \sups (i)._i = \bM B_i, \mathcal{W}_i^{(i)}.^\epsilon$. Then, we have
$$
v_i=\vol(\mathcal{W}\sups(i)._i,\K \mathcal{W}^{(i)}_i. + \Psi \sups (i)._i) = \vol(\mathcal{W}\sups(i)._1,\K \mathcal{W}^{(i)}_1. + \Psi \sups (i)._1).
$$
The second equality holds by \cite[Theorem 1.9.2]{HMX}, while the first one follows from \cite[proof of Corollary 2.1.3]{HMX}. Notice that in \cite[proof of Corollary 2.1.3]{HMX} $\bM.$ is used, while here we are free to use $\bM.^\epsilon$, as all pairs are $\epsilon$-log canonical.

Thus, up to replacing $(X_i,B_i)$ with $(\mathcal{W}\sups(i)._1,\Psi \sups (i)._1)$, we may assume that we are in the situation of Proposition \ref{finite over fixed model}. In particular, $(\Xf_1,\Bf_1)$ plays the role of the model $(Z,D)$ appearing in the statement. By abusing notation, we will denote $(\Xf_1,\Bf_1)$ by $(Z,D)$.

Then, by Proposition \ref{finite over fixed model}, the $v_i$'s can attain just finitely many values. Thus, we get a contradiction. Therefore, the claim follows.
\end{proof}

We conclude by discussing the approach to Theorem \ref{teo discrete} suggested by Chen Jiang.

\begin{remark} \label{remark alternative}
As discussed in the proof of Theorem \ref{teo accumulation}, it is enough to show that the $\mathfrak{V}(d,\Lambda,\epsilon) \cap [0,v]$ is finite for every $v > 0$. Fix $v > 0$, and let $\mathfrak{D}(d,\Lambda,\epsilon,v)$ be the set of pairs $(X,B) \in \mathfrak{D}(d,\Lambda,\epsilon)$ such that $\vol(X,\K X. + B) \leq v$. 

By the same arguments as in the proof of Theorem \ref{teo accumulation}, $\mathfrak{D}(d,\Lambda,\epsilon,v)$ is a log birationally bounded family of pairs. In particular, the corresponding log canonical models form a log birationally bounded family. Notice that, by the assumptions on the pairs $(X,B)$, the corresponding log canonical models are $\epsilon$-log canonical.

Now, by \cite[Theorem 1.6]{HMX14b}, the log canonical models of the pairs in the set $\mathfrak{D}(d,\Lambda,\epsilon,v)$ form a log bounded family. Then, by log boundedness of the family of pairs and finiteness of $\Lambda$, we conclude that $\mathfrak{V}(d,\Lambda,\epsilon) \cap [0,v]$ is finite.
\end{remark}

\printbibliography

\Addresses

\end{document}